\documentclass{amsart}
\usepackage{amsfonts,amssymb,amscd,amsmath,enumerate,verbatim,calc,thumbpdf,mathrsfs,cite,amsthm}
\newtheorem{theo}{Theorem}[section]

\newtheorem{prop}[theo]{Proposition}
\newtheorem{lemma}[theo]{Lemma}
\newtheorem{cor}[theo]{Corollary}
\newtheorem{remark}{Remark}
\newtheorem{definition}[theo]{Definition}
\newtheorem*{thm*}{Theorem}

\begin{document}

\title{Fractal dimensions of subfractals induced by sofic subshifts}
\author{Elizabeth Sattler}
\address{North Dakota State University, PO Box 6050, Fargo, ND 58108}
\email{elizabeth.sattler.1@ndsu.edu}

\maketitle

\begin{abstract}
In this paper, we will consider subfractals of hyperbolic iterated function systems which satisfy the open set condition.  The subfractals will consist of points associated with infinite strings from a subshift of finite type or sofic subshift on the symbolic space.  We find that the zeros of the lower and upper topological pressure functions are lower and upper bounds, respectively, for the Hausdorff, packing, lower and upper box dimensions of the subfractal.
\end{abstract}

\section{Introduction}
One area of interest in fractal geometry is the study of properties which distinguish two distinct fractals.  In particular, fractal dimensions, such as Hausdorff, box, and packing dimensions, have proven to be useful properties that  in a sense, extend our usual notion of topological dimension. 
Numerous results exist for calculating the exact value of fractal dimensions of certain fractals of IFS type, such as self-similar IFSs [4,5,7], or finding bounds for the fractal dimensions for hyperbolic IFSs [2,8].  In this paper, we will focus on specific subsets of fractals of IFS type, namely, subfractals.  

\medskip
Clearly, not every subset of a fractal exhibits fractal-like properties; hence, one must provide a precise definition of a subfractal to produce a genuinely different fractal.  For example, a subset of an IFS fractal may be a contracted copy of the entire fractal, which inherits most of the important properties (including fractal dimensions) from the whole fractal. 

\medskip
In this paper, we have chosen to identify a subfractal of an IFS type fractal  by only considering those points associated with a subshift on the associated symbolic space.   Unless stated otherwise, a subfractal will refer to a subshift-type subfractal for the remainder of this paper.  If a subshift of finite type (SFT) or sofic subshift is chosen, we find that if the subshift does not have full Hausdorff dimension, then the Hausdorff dimension of the subfractal is strictly less than the Hausdorff dimension of the original fractal. 

\medskip
Let $\mathcal{A}$ denote a finite alphabet, $X$ denote the full shift, and $Y \subset X$ be a subshift.  If $Y$ is a SFT, then there exists a matrix $A$ consisting only of 0's and 1's which is associated with the subshift. The entries of $A$ are determined by finite strings in the subshift which are either allowed or not allowed to appear.  
We will use the notation $Y=X_A$ for the subshift associated with $A$.  

\medskip
 Let $\mathcal{K}\subset \mathbb{R}^n$ be a compact subset and $(\mathcal{K}; f_1, \ldots, f_n\}$ be an IFS with $f_i : \mathcal{K} \to \mathcal{K}$ for $1\leq i \leq n$.  The IFS $(\mathcal{K}; f_1, \ldots, f_n\}$ is {\it hyperbolic} if for all $x,y \in \mathcal{K}$, there exists some constant $c$ such that $d(f_i(x), f_i(y)) \leq c d(x,y)$, for all $1\leq i \leq n$.  Let $F$ denote the attractor of this  hyperbolic IFS (HIFS), i.e. $F$ is a non-empty, closed set with $f_i(F) \subset F$ for $1 \leq i \leq n$ and the smallest such set that satisfies these properties.  By [5], we know that an attractor exists for any HIFS. 
Recall that an IFS satisfies the open set condition (OSC) if there exists a nonempty open subset $U \subset \mathcal{K}$ such that $f_i(U) \subset U$ and $f_i(U) \cap f_J(U)  = \emptyset$ for $i \neq j$ and all $ 1 \leq i,j, \leq n$. 

\medskip
Now, let $\mathcal{F}_{X_A}$ be the collection of all points from the full fractal which are associated with a sequence in $X_A$, i.e. $x \in \mathcal{F}_{X_A}$ if there exists some $\omega= \omega_1 \omega_2 \ldots \in X_A$ such that $\displaystyle x =\lim_{k \to \infty} [f_{\omega_k} \circ f_{\omega_{k-1}} \circ \cdots \circ f_{\omega_1}(y)]$, $y \in \mathcal{K}$.  Let $\rho(A)$ denote the spectral radius of a square matrix $A$.  We prove the following:

\begin{thm*}[Main Theorem A]
For a compact subset $\mathcal{K} \subset \mathbb{R}^n$, let $\{\mathcal{K}; f_i : 1 \leq i \leq m\}$ be an HIFS which satisfies the OSC.  Let $0 < c_i \leq  \bar{c}_i < 1$ denote the constants such that $c_id(x,y) \leq d(f_i(x), f_i(y)) \leq \bar{c}_id(x,y)$ for all $1 \leq i \leq m$. Let $X_A$ be an SFT with an associated irreducible square $(0,1)$-matrix $A$. Let $\mathcal{F}_{X_A}$ denote the subfractal associated with the subshift.  Then,
$$ h \leq \dim_H(\mathcal{F}_{X_A}) \leq H \text{ and } h \leq  \overline{\dim}_B(\mathcal{F}_{X_A}) \leq H,$$
where $\rho(AS^{(h)}) = 1 = \rho(A\bar{S}^{(H)})$, $S$ and $\bar{S}$ are the corresponding diagonal matrices with appropriate constants $c_i$ and $\bar{c_i}$ on the diagonal and 0's elsewhere, $(S^{(h)})_{ij} = [S_{ij}]^h$ for all $1\leq i,j \leq N$ and $\bar{S}^{(H)}$ is defined similarly.
\end{thm*}

\medskip
Next, we turn our attention to a broader class of subshifts, the sofic subshifts.  The class of sofic subshifts not only contains all SFTs but also all factors of SFTs.  A common example of a sofic subshift which is not an SFT is the Golden Mean Shift, which has forbidden word list $F =\{101, 10001, \ldots, 10^{2k+1}1, \ldots\}$ on the alphabet $\mathcal{A} = \{0,1\}$.  

\medskip
As in the case of subfractals induced by SFTs, a subfractal induced by a sofic subshift can be represented by a matrix $A_{\mathcal{G}}$; however, the entries of $A_{\mathcal{G}}$ consist of sums of contractive factors associated with finite, allowable strings from the subshift determined by an underlying labeled graph $\mathcal{G}$.  Hence, we must alter the techniques we used for SFTs to compensate for the differences in the matrices associated with the subshifts.  Let $\mathcal{L}$ denote the labeling with the lower contractive bounds and $\bar{\mathcal{L}}$ denote the labeling with the upper contractive bounds.  See Section 5 for details on this labeling.  We prove the following:

\begin{thm*}[Main Theorem B]
For a compact subset $\mathcal{K} \subset \mathbb{R}^n$, let $\{\mathcal{K}; f_i : 1 \leq i \leq m\}$ be an HIFS which satisfies the OSC.  Let $0 < c_i \leq  \bar{c}_i < 1$ denote the constants such that $c_id(x,y) \leq d(f_i(x), f_i(y)) \leq \bar{c}_id(x,y)$ for all $1 \leq i \leq m$.
 Let $X_{\mathcal{G}}$ be a sofic subshift wtih irreducible matrices  $A =(a_{ij})_{1\leq i,j \leq k}$ and $\bar{A} =(\bar{a}_{ij})_{1\leq i,j \leq k}$ with $a_{ij} = \sum_{e_{ij}} \mathcal{L}(e_{ij})$ and $\bar{a}_{ij} = \sum_{e_{ij}} \bar{\mathcal{L}}(e_{ij})$.
Then,
$$ h \leq \dim_H(\mathcal{F}_{X_\mathcal{G}}) \leq H \text{ and } h \leq  \overline{\dim}_B(\mathcal{F}_{X_\mathcal{G}}) \leq H,$$
where $\rho(A_{ h}) = 1 = \rho(A_H)$, and $A_h$, $A_{H}$ have entries $a_{ij}^{(h)} = \sum_{e_{ij}} \mathcal{L}(e_{ij})^h$,  $a_{ij}^{(H)} = \sum_{e_{ij}} \bar{\mathcal{L}}(e_{ij})^H$, respectively.
\end{thm*}

\begin{remark} 
Theorem A will be split into Theorems 4.6 and 4.7 (for Hausdorff dimension bounds and upper box dimension bounds, respectively) in the case where $A$ is an irreducible matrix.  Similarly, the case in which matrix $A_{\mathcal{G}}$ from Theorem B is irreducible will be presented as Theorem 5.2.  Theorem 6.3 will extend the results for Hausdorff dimension of subfractals defined by either a SFT or sofic subshift with a reducible matrix.
\end{remark}

\medskip
These results generalize previously proven results, including Theorems 1.1 and 1.2 below, which analyze different types of subfractals [2,8].
Following the notation and terminology in [8], we say that $A$ is {\it primitive} if there exists some integer $N$ such that $(A^N)_{ij} > 0$ for all $1 \leq i,j \leq n$, where $(A^N)_{ij}$ denotes the $ij$-entry of $A^N$.  
A sequence of integers $(i_l)_{l \geq 1}$, where $i_l \in \{1, \ldots, n\}$, is said to be {\it admissible} if $(A)_{i_l,i_{l+1}} \neq 0$ for all $l \geq 1$.  
Let $F_A$ denote the collection of all points in $F$ which are associated with an admissible sequence with respect to $A$.

\medskip
An HIFS is called {\it disjoint} if $f_i(F) \cap f_j(F) = \emptyset$ for all $i \neq j$ and $1 \leq i,j \leq n$.  In [2], Ellis and Branton proved the following theorem.
\begin{theo}
Let $F$ be the attractor of a disjoint HIFS $(\mathcal{K}; f_1, \ldots, f_n)$, and let $A$ be a primitive (0,1)-matrix.   Suppose that 
$$ s_i d(x,y) \leq d(f_i(x), f_i(y)) \leq \bar{s_i} d(x,y),$$
for all $x,y \in \mathcal{K}$, $1 \leq i \leq n$, and for some constants $0 < s_i \leq \overline{s_i} <1$.  Then, $\dim_H(F_A) \leq u$, where $\rho(A\bar{S}^u)=1$ and $\bar{S}$ is the diagonal matrix with $\text{diag}(\bar{s_1}, \ldots, \bar{s_n})$.

\end{theo}
\medskip
In the same paper [2], Ellis and Branton made the following conjecture for the lower bound: $\dim_H(F_A) \geq l$ where $\rho(AS^l)=1$ and $S$ is a diagonal matrix with $s_1, \ldots, s_n$ on the diagonal and zeros elsewhere.

\medskip
An $n \times n$ matrix $A$ is called {\it irreducible} if for all $1\leq i,j \leq n$, there exists some finite sequence $(i_l)_{1 \leq l \leq m}$ with $i=i_1$ and $j=i_m$  such that $(A)_{i_l, i_{l+1}} (A)_{i_{l+1}, i_{l+2}} > 0$ for $1 \leq l \leq m$. 
Every primitive matrix is irreducible, but there exist matrices which are irreducible and not primitive [6].
 Let $N \geq 2$ and $\{\mathcal{K} ; f_{ij}, (A)_{ij} : 1 \leq i \leq N \}$, where $f_{ij}:\mathcal{K} \to \mathcal{K}$ is a hyperbolic map for $1 \leq i,j \leq N$ and $A$ is an irreducible $(0,1)$-matrix.  
The system $\{\mathcal{K} ; f_{ij}, (A)_{ij} : 1 \leq i \leq N \}$ is called a {\it hyperbolic recurrent IFS}.

\medskip
A particular case of Roychowdhury's result below not only proves the conjecture proposed by Ellis and Branton, but also generalizes Theorem 1.1 by allowing the matrix $A$ to be irreducible and  requiring the IFS to satisfy the OSC:
\begin{theo}
Let $\{\mathcal{K}; f_{ij}, (A)_{ij} : 1 \leq i,j  \leq N\}$ be a hyperbolic recurrent IFS which satisfies the open set condition and assume $A$ is irreducible.  Let $F_A$ be the attractor of the system.  Then,
$$ h \leq \dim_H(F_A) \leq H \text{ and } h \leq \overline{\dim}_H(F_A) \leq H,$$
where $h$ and $H$ are given by $\rho(((A)_{ij}s_{ij}^h)_{1\leq i,j \leq N}) = 1$ and $\rho(((A)_{ij} \bar{s}_{ij})_{1 \leq i,j \leq N })$ and $s_{ij}, \bar{s}_{ij}$ are given by $s_{ij}d(x,y) \leq d(f_{ij}(x), f_{ij}(y)) \leq \bar{s}_{ij} d(x,y)$.
\end{theo}

\medskip

\medskip
Although it was not stated so, an attractor described above in Theorems 1.1 and 1.2 can be associated with an SFT defined by a list of forbidden words, each of length 2. 
Theorem A generalizes Theorem 1.1 completely in $\mathbb{R}^n$ and partially generelizes Theorem 1.2 by allowing the subfractal to be associated with any SFT, regardless of the length of the words in the forbidden word list.  Furthermore, we extend the results to subfractals induced by a sofic subshift, which is a broader class than SFTs and, to our knowledge, is new.  In the case of Hausdorff dimension, we remove the requirement that the associated matrices must be irreducible, and hence our results include even more subfractals induced by SFTs and sofic subshifts.

\medskip

\section{Basic definitions and background}
Let $\mathcal{K} \subset \mathbb{R}^n$ be a compact subset and $E \subseteq \mathcal{K}$.
Letting $\displaystyle \overline{\mathcal{H}}^s_{\varepsilon} (E) = \inf_{\mathcal{U} \in \mathcal{O}} \sum_{U \in \mathcal{U}} (\text{diam}(U))^s$, where $\mathcal{O}$ is the collection of all open $\varepsilon$-covers of $E$ and $s \geq 0$,  
the $s$-dimensional Hausdorff outer measure is defined to be $\displaystyle \overline{\mathcal{H}}^s = \lim_{\varepsilon \to 0} \overline{\mathcal{H}}^s_{\varepsilon}$.    
Restricting the outer measure to measurable sets, one defines the $s$-dimensional Hausdorff measure, $H^s$.  The {\it Hausdorff dimension of $E$}, denoted $\dim_{\text H}(E)$, is defined as the unique value of $s$ such that:
$$ H^r(E) = \left \{
\begin{array}{lr}
0 ,& r > s \\
\infty ,& r < s .
 \end{array} \right . $$

\medskip
If $N_r(E)$ denotes the smallest number of sets of diameter $r$ that can cover $E$, the {\it lower and upper box dimensions of $E$} are defined, respectively, as [3]:
$$ \underline{\dim}_\text{B}(E)= \liminf_{r \to 0} \frac{\log N_r(E)}{-\log r} \text{ and }
\overline{\dim}_\text{B}(E) = \limsup_{r \to 0} \frac{\log N_r(E)}{-\log r}. $$
The following relationship between the fractal dimensions defined above are well-known [3]:
$$ \dim_H(E) \leq \underline{\dim}_B(E) \leq \overline{\dim}_B(E).$$

\medskip
Let $\mathcal{A} = \{1, \ldots, m\}$ be a finite alphabet.  Let $\Omega_n$ denote the collection of all words on $\mathcal{A}$ of length $n$ and $\Omega_*= \displaystyle \bigcup_{k=1}^\infty \Omega_k$ denote the collection of all finite words of any finite length. 
 Let $X$ denote the compact metric space of all infinite sequences on $\mathcal{A}$, equipped with the metric $d_X$ defined by $d_X(\omega, \tau) = \displaystyle \frac{1}{2^k}$ where $k = \min \{ i : \omega_i \neq \tau_i\}$, for all $\omega = \omega_1\omega_2 \ldots, \tau = \tau_1 \tau_2 \ldots \in X$. 
For $\omega \in \Omega_*$, let $\ell(\omega)$ denote the length of the word $\omega$.   Let $\sigma : X \to X$ denote the shift map defined by $\sigma(\omega_1 \omega_2 \ldots) = \omega_2 \omega_3 \ldots$ for all $\omega = \omega_1 \omega_2 \ldots \in X$.   We will also adopt the following notations:
\begin{align*}
\omega \tau &= \omega_1 \ldots \omega_n \tau_1 \ldots \tau_m \text{ for } \omega \in \Omega_n, \tau \in \Omega_m, \\
\vspace{.05in}
\omega^- &= \omega_1 \ldots \omega_{n-1} \text{ for } \omega \in \Omega_n,\\
\vspace{.05in}
\omega|_n &= \omega_1 \ldots \omega_n \text{ for all } \omega \in X. 
\end{align*}

\medskip
We will begin by focusing on specific subshifts, namely, subshifts of finite type (SFTs).  An SFT, say Y, is defined by a finite list of forbidden words of finite length.  A word $\tau \in \Omega_n$ is {\it forbidden} if it appears nowhere in $\omega$ for all $\omega \in Y$. 
Any word that is not forbidden is called an {\it allowable} word.
Observe that for any list of forbidden words $F = \{x_1, \ldots, x_k\}$, $x_i \in \Omega_*$ for $1 \leq i \leq k$, there exists an integer $N$ such that $F$ can be rewritten as $F = \{y_1, \ldots, y_l\}$ where $y_i \in \Omega_N$ for $1 \leq i \leq l$.  For more information on SFTs, refer to [6]. 

\medskip
Let $\omega = \omega_1 \ldots \omega_{k-1}, \xi = \xi_1 \ldots \xi_{k-1} \in \Omega_{k-1}$.  We say $\omega$ is {\it compatible with $\xi$} if $\omega_2 \ldots \omega_{k-1} = \xi_1 \ldots \xi_{k-2}$.  A compatible pair is a pair $(\omega, \xi) \in \Omega_{k-1} \times \Omega_{k-1}$, where $\omega$ is compatible with $\xi$.
Let $(\Omega_{k-1} \times \Omega_{k-1})_{\text comp}$ denote the collection of all compatible pairs $(\omega, \xi) \in \Omega_{k-1} \times \Omega_{k-1}$.  Define an operation $* : (\Omega_{k-1} \times \Omega_{k-1})_{\text comp} \to \Omega_k$ by $\omega * \xi = \omega_1 \omega_2 \ldots \omega_{k-1} \xi_{k-1} = \omega_1 \xi_1 \xi_2 \ldots \xi_{k-1}$.

\medskip
Let $X_F$ be a SFT with forbidden words $F = \{ \tau_1, \ldots, \tau_l\}$.  Without loss of generality, we can assume $\tau_i \in \Omega_k$ for all $1\leq i \leq l$.  Let $W_n(X_F)$ denote all allowable words of length $n$ from $X_F$ for $n \geq 1$.  If the subshift $X_F$ is clearly understood in context, we will typically write $W_n$.  Let $W_* = \displaystyle \bigcup_{k=1}^\infty W_k$ denote the collection of all finite allowable strings.

\medskip
 Let $N = m^{k-1}$, where $m = |\mathcal{A}|$ and $\ell(\tau_i)=k$ for all $\tau_i \in F$. 
We will construct an $N$ x $N$ adjacency matrix $A$ as follows.  Label the rows with all possible words (both allowable and forbidden) of length $k-1$, i.e. label the rows with $\{\omega_1, \ldots, \omega_N\} = \Omega_{k-1}$.  Label the correpsonding columns similarly.  Let the entry be $a_{ij} = 0$ if $\omega_i$ is not compatible with $\omega_j$ and $a_{ij} = 0$ if $\omega_i$ is compatible with $\omega_j$ but $\omega_i * \omega_j \in F$.  The entry $a_{ij} = 1$ if $\omega_i$ is compatible with $\omega_j$ and $\omega_1 * \omega_j$ is an allowable word.

\medskip
For the sake of clarity, consider the following examples.  First, consider the SFT on the alphabet $\mathcal{A} = \{1,2\}$ with forbidden word $F_1 = \{22\}$.  The associated matrix will be of the form:

$$\begin{bmatrix} 1&1 \\ 1&0 \end{bmatrix}.$$

\medskip
Next, let us consider a SFT on the same alphabet $\mathcal{A} = \{1,2\}$ but with forbidden word list $F_2 = \{ 112, 211, 222\}$.  Since each forbidden word has length 3, we will need to consider a 4 x 4 matrix since $|\Omega_2| = 4$.  We will choose the following labeling of rows: $R_1 \rightarrow 11, R_2 \rightarrow 12, R_3 \rightarrow 21, R_4 \rightarrow 22$.  The corresponding matrix will be of the form:

$$\begin{bmatrix} 1&0&0&0 \\ 0&0&1&1 \\ 0&1&0&0 \\ 0&0&1&0 \end{bmatrix}.$$

Here, the entries $a_{12} = a_{31} = a_{44} = 0$ correspond to the forbidden words 112, 211, 222, respectively.  
The entries $a_{13} = a_{14} = a_{21} = a_{22} = a_{33} = a_{34} = a_{41} = a_{42}=0$ correspond to pairs which are not compatible.  The 1's in the matrix all correspond to compatible pairs which are also allowable words.  We will use either $X_A$ or $X_F$ to denote the SFT.

\medskip
To each such $N$ x $N$ adjacency matrix, we can associate a directed graph $G_A = (V_A, E_A)$ where $V = \{v_1, v_2, \ldots, v_N\}$ and $E = \{e_{i,j}\}_{i,j = 1}^N$ where $e_{i,j}$ is an edge from $v_i$ to $v_j$ if the entry $a_{ij} = 1$ from $A$.   A directed graph $G = (V,E)$ is called {\it strongly connected} if for any two vertices $v_i, v_j \in V$, there exists a path from $v_i$ to $v_j$.  

\begin{prop} 
A matrix $A$ is irreducible iff it is associated with a graph $G_A$ which is strongly connected.
\end{prop}

For details on Proposition 2.1, see [6].
By the Perron-Frobenius Theorem, we know that if $A$ is an irreducible matrix, then $A$ has a positive eigenvector ${\bf v}_A$  corresponding to a positive eigenvalue $\lambda_A \in \mathbb{R}$ such that $|\mu| \leq \lambda_A$ where $\mu$ is any eigenvalue of $A$ [6].
For any non-negative $m$ x $m$ matrix $A$ with a positive eigenvector and corresponding positive eigenvalue $\lambda$, there exist constants $b_0, d_0 > 0$ such that
 $$b_0 \lambda^n \leq \displaystyle \sum_{i,j = 1}^m (A^n)_{ij} \leq d_0 \lambda^n.$$

\section{Subfractals associated with a Subshift}
Let $\{\mathcal{K}; f_1, \ldots f_m\}$ be the system defined in the statement of the main theorem, and let $\mathcal{F}$ denote the attractor of the HIFS.  If $\mathcal{A} = \{1, \ldots, m\}$, where each letter $i$ corresponds to the map $f_i $ for $1 \leq i \leq m$,  and $\omega = \omega_1 \ldots \omega_n \in \Omega_n$, we will use the following notation:
\begin{align*}
f_\omega &= f_{\omega_n} \circ f_{\omega_{n-1}} \circ \cdots \circ f_{\omega_1} \\ \vspace{.08 in}
c_\omega &= c_{\omega_1} c_{\omega_2} \cdots c_{\omega_n}.
\end{align*}
Define the associated coding map $\pi : X \to \mathcal{F}$ by $\pi ( \omega) = \displaystyle \lim_{n \to \infty} f_{\omega|_n}(\mathcal{K})$. 

\medskip
For each such IFS, we can define a subfractal of $\mathcal{F}$ by only considering the points associated with a word from a subshift.  Let $X_F$ be a SFT and define $\mathcal{F}_{X_F} = \{\pi(\omega) : \omega \in X_F\}$.  

As defined in section 2, fix an $N \times N$ adjacency matrix.  Let $\Omega_{k-1} = \{\tau^1, \tau^2, \ldots, \tau^N\}$, $N=m^{k-1}$.  
Define two other $N \times N$ matrices, $S_0$ and $S$, as follows:
$$S_0 = \begin{bmatrix}
c_{\tau^1}& 0 &\cdots & 0 \\
0 & c_{\tau^2} &\cdots & 0 \\
\vdots & \vdots & \ddots & \vdots \\
0& 0 & \cdots& c_{\tau^N}.
\end{bmatrix}
\text{ and }
S = \begin{bmatrix}
c_{i_1}& 0 &\cdots & 0 \\
0 & c_{i_2} &\cdots & 0 \\
\vdots & \vdots & \ddots & \vdots \\
0& 0 & \cdots& c_{i_N}
\end{bmatrix},$$

where $i_j \in \mathcal{A}$ for all $1 \leq j \leq N$ and the order of the $i_j's$ is chosen so that 
$$\sum_{i,j =1}^N (S_0 A_0 S)_{i,j} = \sum_{\omega \in \Omega_{k-1}} c_\omega,$$
with adjacency matrix $A_0$ associated with the full shift.
Similarly, we define
$$\bar{S}_0 = \begin{bmatrix}
\bar{c}_{\tau^1}& 0 &\cdots & 0 \\
0 & \bar{c}_{\tau^2} &\cdots & 0 \\
\vdots & \vdots & \ddots & \vdots \\
0& 0 & \cdots& \bar{c}_{\tau^N}
\end{bmatrix}
\text{ and }
 \bar{S} = \begin{bmatrix}
\bar{c}_{i_1}& 0 &\cdots & 0 \\
0 & \bar{c}_{i_2} &\cdots & 0 \\
\vdots & \vdots & \ddots & \vdots \\
0& 0 & \cdots& \bar{c}_{i_N}
\end{bmatrix}.
$$
For $t \in \mathbb{R}$, define
$$ S^{(t)} = \begin{bmatrix}
c_{i_1}^t& 0 &\cdots & 0 \\
0 & c_{i_2}^t &\cdots & 0 \\
\vdots & \vdots & \ddots & \vdots \\
0& 0 & \cdots& c_{i_N}^t
\end{bmatrix},
$$
and define $S_0^{(t)}$, $\bar{S}^{(t)}$, and $\bar{S}_0^{(t)}$ similarly.\\

Next, we will define a topological pressure function for calculating bounds for the fractal dimensions.  Topological pressure functions have been used to find bounds for fractal dimensions of different types of fractal classes [8].

\begin{definition} Let $X_A$ be a subshift.  The {\it  lower topological pressure function} of $\mathcal{F}_{X_A}$ is given by $P(t) = \displaystyle \lim_{n \to \infty} \frac{1}{n} \log \left( \sum_{\omega \in W_n} c_{\omega}^t \right )$.
Similarly, we define the {\it upper topological pressure function} by $\bar{P}(t) = \displaystyle \lim_{n \to \infty} \frac{1}{n} \log \left (\sum_{\omega \in W_n} \bar{c}_\omega^t \right )$.
\end{definition} 

\begin{prop} The lower and upper topological pressure functions $P(t)$ and $\bar{P}(t)$ are strictly decreasing, convex, and continuous on $\mathbb{R}. $
\end{prop}
\begin{proof}
We will show the proof for $P(t)$. The proof for $\bar{P}(t)$ follows similarly.
Let $\delta > 0$.  By using the fact that $c_{\omega} \leq  c_{max}^n$ for all $\omega \in W_n$, where $c_{max} = \max_{1\leq i \leq m} \{c_i\}$, we have:  \\
\begin{align*}
 P(t + \delta) & =  \lim_{n \to \infty} \frac{1}{n} \log \left( \sum_{\omega \in W_n}  c_\omega^{t + \delta} \right) \leq \lim_{n \to \infty} \frac{1}{n} \log \left( \sum_{\omega \in W_n}  c_\omega^{t}c_{max}^{n\delta} \right) \\
& = \lim_{n \to \infty} \frac{1}{n} \log \left( c_{max}^{n\delta}\sum_{\omega \in W_n}  c_\omega^{t}\right)  = \lim_{n \to \infty} \frac{1}{n} [ n\delta \log(c_{max})] + P(t) \\
& = \delta \log(c_{max}) + P(t) < P(t),  \text{ since }  0<c_{max}<1 .
\end{align*}
Hence, $P(t)$ is strictly decreasing. 
 If $t_1, t_2 \in \mathbb{R}$ and $a_1, a_2 >0$ with $a_1+a_2 = 1$,  then, by H\"older's inequality, we have\\
\begin{align*}
P(a_1t_1 + a_2t_2)& = \lim_{n \to \infty} \frac{1}{n} \log \left( \sum_{\omega \in W_n} (c_\omega)^{a_1t_1+a_2t_2} \right) \\
&= \lim_{n \to \infty} \frac{1}{n} \log \left[ \sum_{\omega \in W_n} ((c_\omega)^{t_1})^{a_1} ((c_\omega)^{t_2})^{a_2}\right] \\
&  \leq \lim_{n \to \infty} \frac{1}{n} \log\left[\sum_{\omega \in W_n} (c_\omega)^{t_1} \right]^{a_1} \left[\sum_{\omega \in W_n} (c_\omega)^{t_2}\right]^{a_2} \\
&= a_1 P(t_1) + a_2 P(t_2) . 
\end{align*}
Hence, $P(t)$ is a convex function and strictly decreasing, and thus must be continuous.\\
\end{proof}

\begin{prop}
There is a unique value $h \in [0, \infty)$ such that $P(h) = 0$.
\end{prop}
\begin{proof}
If $t = 0$,\\
$$ P(0) = \lim_{n \to \infty} \log \left ( \sum_{\omega \in W_n} c_\omega^0 \right ) = \lim_{n \to \infty} \log (|W_n|) \geq 0.$$
Next, we will look at the case where $t \to \infty$.\\
\begin{align*}
P(t) &= \lim_{n \to \infty} \frac{1}{n} \log \left( \sum_{\omega \in W_n} c_\omega^t \right ) \leq \lim_{n \to \infty} \frac{1}{n}\log \left ( \sum_{\omega \in W_n} c_{max}^{nt} \right ) \\
& = t \log(c_{\max}) + \lim_{n \to \infty} \frac{1}{n} \log( |W_n| ) \leq t\log(c_{max}) + \lim_{n \to \infty} \frac{1}{n} \log(m^n)\\
& = t\log(c_{max}) + \log(m).
\end{align*}
Since $0 < c_{max} < 1$, we must have $[t \log(c_{max}) + \log(m)] \to -\infty$ as $t \to \infty$, and hence $ \displaystyle \lim_{t \to \infty} P(t) = - \infty$. 
By Proposition 3.2, there exists a unique value $h$ such that $P(h)=0$.
\end{proof}
Following the same steps as in the proof above, we have:

\begin{prop} There is a unique value $H \in [0, \infty)$ such that $\bar{P}(H) = 0.$
\end{prop}

\begin{prop}
Let $h$ and $H$ be the unique values such that $P(h) = 0 = \bar{P}(H)$.  Then, $h \leq H$.
\end{prop}
\begin{proof} Assume that $h > H$.  Then, $\bar{P} (h) < \bar{P}(H) = 0$.  
We also know that $c_\omega \leq \bar{c}_\omega$ for all $\omega \in W_n$.  Hence,
$$ 0 = P(h) = \lim_{n \to \infty} \frac{1}{n} \log \left (\sum_{\omega \in W_n} c_\omega^h \right ) \leq \lim_{n \to \infty} \frac{1}{n} \log \left ( \sum_{\omega \in W_n} \bar{c}_\omega ^h \right ) = \bar{P} (h) < 0,$$
which is a contradiction.  Hence, $h \leq H$.
\end{proof}
 
\begin{lemma} Let $X_A$ be  an SFT associated with matrix $A$, and let $\{\mathcal{K}; f_i : 1 \leq i \leq m\}$ be an HIFS.  Let $S_0$ and  $S$ be matrices associated with the subfractal $\mathcal{F}_{X_A}$, as above. Then, the associated lower and upper topological pressure functions $P(t)$ and $\bar{P}(t)$ can be written,respectively, as 
$$P(t) = \displaystyle \lim_{n \to \infty} \frac{1}{n} \log \left ( \sum_{i,j=1}^N [S_0^{(t)} (AS^{(t)})^{n-k+1}]_{i,j} \right ) \text{ and }$$
$$\bar{P}(t) = \displaystyle \lim_{n \to \infty} \frac{1}{n} \log \left ( \sum_{i,j=1}^N [\bar{S}_0^{(t)} (A\bar{S}^{(t)})^{n-k+1}]_{i,j} \right).$$
\end{lemma}
\begin{proof}
Recall that if $F$ is a list of forbidden words, all of length $k$, then $A$ is an $N$ x $N$ matrix, where $N = |\Omega_{k-1}|=m^{k-1}$. We will prove the assertion by induction.  First, the nonzero entries of $A$ correspond to the allowable words of length $k$.  Hence, by definition of $A, S_0, \text{ and } S$, we have 
$$ \sum_{i,j =1}^N [S_0 A S]_{ij} = \sum_{\omega \in W_k} c_\omega.$$
Now, assume that $\displaystyle \sum_{i,j=1}^N [S_0(AS)^n]_{ij} = \sum_{\omega \in W_{n+k-1}}c_\omega$ for some $n>1$.  The entries of $S_0(AS)^n$ consist of sums of contractive factors associated with allowable words of length $n+k-1$.  Now, consider the matrix $ S_0(AS)^n (AS)$.  
By the definition of $A$ and $S$, this multiplication will result in entries consisting of sums of contractive factors associated with allowable words of length $n+k$.  Since $S_0(AS)^n$ contains all allowable words of length $n+k-1$, then we must have $\displaystyle \sum_{i,j=1}^N [S_0(AS)^{n+1}]_{ij} = \sum_{\omega\in W_{n+k}} c_\omega.$  
Hence, 
$$P(t) = \lim_{n \to \infty} \frac{1}{n} \log \left (\sum_{\omega \in W_n} c_\omega^t \right ) = \lim_{n \to \infty} \frac{1}{n} \log \left (  \sum_{i,j=1}^N [S_0^{(t)}(AS^{(t)})^{n-(k-1)}]_{ij}\right ).$$
The proof follows similarly for the upper topological pressure function.

\end{proof}

\section{Main theorem for SFTs}
We begin with a technical lemma that will provide bounds needed for the main result.
\begin{lemma}
Let $S_0$, $A$, and $S$ be defined as in Section 3, where $A$ is an irreducible matrix.  
Then, for any $t > 0$, there exist positive constants $K, L$ such that
$$ c_{min}^{(k-1)t} K \lambda_{AS^{(t)}}^n \leq  \sum_{i,j=1}^N [S_0^{(t)} (AS^{(t)})^n]_{i,j} \leq c_{max}^{(k-1)t} L \lambda_{AS^{(t)}}^n,$$

 where $c_{min} = \displaystyle \min_{1\leq i \leq m} \{c_i\}$, $c_{max} = \displaystyle \max_{1 \leq i \leq m} \{ c_i\}$, $\lambda_{AS^{(t)}}$ is the maximal eigenvalue of $AS^{(t)}$.
\end{lemma}
\begin{proof}
Notice that for every non-zero entry of $S_0$, we have $c_{min}^{k-1} \leq (S_0)_{ij} \leq c_{max}^{k-1}$, $1 \leq i,j \leq N$.  Hence, by the Perron-Frobenius Theorem, we have constants $K$ and $L$ such that
\begin{center}
$ \displaystyle c_{min}^{(k-1)t} K \lambda_{AS^t}^n \leq c_{min}^{(k-1)t} \sum_{i,j=1}^N [(AS^{(t)})^n]_{i,j} \leq  \sum_{i,j=1}^N [S_0^{(t)} (AS^{(t)})^n]_{i,j}$\\
$ \leq c_{max}^{(k-1)t}  \sum_{i,j=1}^N [(AS^{(t)})^n]_{i,j} \leq c_{max}^{(k-1)t} L \lambda_{AS^{(t)}}^n$
\end{center}
\end{proof}

\begin{remark}
By Lemma 4.1, one can show that, for fixed value $t \in [0,\infty]$,
\begin{align*}
 P(t) &= \lim_{n \to \infty} \frac{1}{n} \log \left (\sum_{i,j=1}^N [S_0^{(t)}(AS^{(t)})^n]_{ij}\right )\\
 &\leq \lim_{n \to \infty} \frac{1}{n} \log (c_{max}^{(k-1)t} L \lambda_{AS^{(t)}}^n) = \log(\lambda_{AS^{(t)}}) = \log(\rho(AS^{(t)}),\end{align*}
where $\rho(AS^{(t)})$ denotes the spectral radius of $AS^{(t)}$.  Similarly, we can show that $\log(\rho(AS^{(t)})) \leq P(t)$, and hence $P(t) = \log(\rho(AS^{(t)})$.  Therefore, the unique value $h$ such that $P(h)=0$ is also the value of $h$ such that $\rho(AS^{(h)}) = 1$.  Analogously, we can show that the value $H$ such that $\bar{P}(H)=0$ is also the value of $H$ that satisfies $\rho(A\bar{S}^{(H)})=1$.

\end{remark}

\begin{prop}
Let $h$ be the unique zero of the lower topological pressure function.  There exist positive constants $K_0, L_0$ such that 
$$ K_0 \leq \sum_{\omega \in W_n} c_{\omega}^h \leq L_0,$$
for all $n \geq 1$.
\end{prop}
\begin{proof}
Let $s<h$.  Then, $P(s) > P(h) = 0$.  So, we have
\begin{align*}
 0 < P(s) &=  \lim_{p \to \infty} \frac{1}{np} \log \left ( \sum_{\omega \in W_{np}} c_{\omega}^s\right) \leq \lim_{p \to \infty} \frac{1}{np} \log \left ( \sum_{\omega \in W_n} c_ {\omega}^s \right )^p\\
 &= \frac{1}{n} \log \left ( \sum_{\omega \in W_n} c_\omega^s \right ).
\end{align*}
Hence, $\displaystyle \sum_{\omega \in W_n} c_\omega^s > 1$, and it follows that $\displaystyle \sum_{\omega \in W_n} c_\omega^h \geq 1$.\\
Now, assume that $s>h$.  Then, $0 = P(h) > P(s)$.  So, by Lemma 4.1, we have
\begin{align*}
0 > P(s) &= \displaystyle \lim_{p \to \infty} \frac{1}{np} \log \left ( \sum_{\omega \in W_{np}} c_{\omega}^s\right) = \lim_{p \to \infty} \frac{1}{np} \log \left (\sum_{i,j=1}^N [S_0^{(s)} (AS^{(s)})^{np}]_{i,j} \right ) \\
&\geq \lim_{p \to \infty} \frac{1}{np} \log \left ( c_{min}^{(k-1)s} K \lambda_{AS^{(s)}}^{np} \right ) = \frac{1}{n} \log (\lambda_{AS^{(s)}}^n)\\
& \geq \frac{1}{n} \log \left (\frac{1}{L c_{max}^{(k-1)s}} \sum_{i,j=1}^N [S_0^{(s)} (AS^{(s)})^n]_{i,j} \right ) 
 = \frac{1}{n} \log \left (\frac{1}{L c_{max}^{(k-1)s}} \sum_{\omega \in W_n} c_\omega^s \right ).
\end{align*}
Hence, $\displaystyle \sum_{\omega \in W_n} c_\omega^s < L c_{max}^{(k-1)s}$, which implies that $\displaystyle \sum_{\omega \in W_n} c_\omega^h \leq L c_{max}^{(k-1)h}$.
\end{proof}

Following similar steps in the proof of Proposition 4.2, we have:
\begin{prop}
Let $H$ be the unique zero of the upper topological pressure function.  There exist positive constants $K_1, L_1$ such that 
$$K_1 \leq \sum_{\omega \in W_n} \bar{c}_\omega^H \leq L_1.$$
\end{prop}

In order to show that $h$ is a lower bound for $\dim_H(\mathcal{F})$, we will utilize the uniform mass distribution principle from Falconer [4].  Hence, we must define an appropriate Borel probability measure to satisfy the principle.  
Let $h$ be the unique value such that $P(h)=0$.   Let $\omega \in \Omega_n$ and let $[\![ \omega ]\!]  =\{ \tau \in \Omega_\infty : \tau_i = \omega_i, 1 \leq i \leq n\}$ be the cylinder set corresponding to $\omega$.  We will use the fact that $c_{\omega \tau} = c_\omega c_\tau$.  Define
$$\nu_n([\![ \omega ]\!] ) =  \frac{\displaystyle \sum_{\omega\tau \in W_{n + \ell(\omega)} }c_{\omega\tau}^h}{\displaystyle \sum_{\tau \in W_{n + \ell(\omega)}} c_{\tau}^h} . $$

\medskip
For all $n \geq 1$ and any $\omega \in W_*$, we have by Proposition 4.2,
$$0\leq \frac{\sum_{\omega\tau \in W_{n+\ell(\omega)}} c_{\omega\tau}^h}{L_0}  \leq \nu_n([\![ \omega ]\!] ) \leq \frac{c_\omega^h\sum_{\tau \in W_n} c_{\tau}^h}{\sum_{\tau \in W_n+\ell(\omega)} c_{\tau}^h} \leq \frac{L_0}{K_0} c_\omega^h < \infty.$$
 Hence, for all $\omega \in W_*$, $\text{Lim}_{n \to \infty} \nu_n([\![ \omega ]\!] )$ exists, where $\text{Lim}$ denotes the Banach limit.  Let $\nu([\![ \omega ]\!] ) = \text{Lim}_{n \to \infty} \nu_n([\![ \omega ]\!] )$.

\medskip
Also, notice that
\begin{align*}
\sum_{i=1}^m \nu([\![ \omega i ]\!] ) &= \text{Lim}_{n \to \infty} \sum_{i=1}^m  \frac{\sum_{\omega i\tau \in W_{n+ \ell(\omega i)}} c_{\omega i \tau}^h}{\sum_{\tau \in W_{n + \ell(\omega i)}} c_{\tau}^h}\\
 &= \text{Lim}_{n \to \infty}  \frac{\sum_{\omega\tau \in W_{n+1 + \ell(\omega)}}c_{\omega  \tau}^h}{\sum_{\tau \in W_{n + 1 +  \ell(\omega)}} c_{\tau}^h} = \nu ([\![ \omega  ]\!] ).
\end{align*}

\medskip
Hence, by applying Kolmogorov extension theorem, we can extend $\nu$ to a unique Borel probability measure $ \gamma$ on $X_A$.
Let $\mu_h = \gamma \circ \pi^{-1}$, where $\pi$ is the coding map.  Hence,  $\mu_h$ is supported on $\mathcal{F}_{X_A}$.

\begin{cor}
There exist constants $K_0, L_0 > 0$ such that 
$$\mu_h(f_\omega(\mathcal{K})) \leq \frac{L_0}{K_0} c_\omega^h.$$
\end{cor}
\begin{proof}

By definition of $\mu_h$ and Proposition 4.2, we have
$$ \mu_h(f_{\omega}(\mathcal{K})) = \nu( [\![ \omega ] \!] ) =  \frac{\displaystyle \sum_{\omega \tau \in W_{n} }c_{\omega\tau}^h}{\displaystyle \sum_{\tau \in W_{n + \ell(\omega)}} c_{\tau}^h} \leq \frac{ \displaystyle c_\omega^h \sum_{\tau \in W_n} c_\tau^h}{\displaystyle \sum_{\tau \in W_{n + |\omega|}} c_{\tau}^h} \leq c_\omega^h \frac{L_0}{K_0}.$$
\end{proof}

\begin{prop}
For $0<r<1$ and $x \in \mathcal{F}$, the ball $B(x,r)$ intersects at most $M$ elements of $\mathcal{U}_r = \{ f_\omega(\mathcal{K}) : |f_\omega(\mathcal{K})| \leq r < |f_{\omega^-}(\mathcal{K})| \}$, where $M$ is finite and independent of $r$.
\end{prop}
\begin{proof}
Let $0<r<1$ and $x \in \mathcal{F}$.  Let $W_r = \{\omega\in W_* : f_\omega(\mathcal{K}) \cap B(x,r) \neq \emptyset, f_\omega(\mathcal{K}) \in \mathcal{U}_r\}$ and $|W_r|=M$.
Let $y \in B(x,r)$ and $z \in f_\omega(\mathcal{K})$ where $\omega \in W_r$.  Notice that 
$$d(y,z) \leq |B(x,r)| + |f_\omega(\mathcal{K})| \leq 3r.$$
Hence, $\{f_\omega(\mathcal{K}) : \omega \in W_r\} \subset B(x,3r)$.  For any $f_\omega(\mathcal{K}) \in \mathcal{U}_r$, we have
$$|f_\omega(\mathcal{K})| \geq  c_{min} |f_{\omega^-}(\mathcal{K})| > c_{min} r.$$
Due to the open set condition, there exists a ball $B_a$ of radius $a>0$ such that $B_a \subset \mathcal{K}$ and $f_\omega(B_a) \cap f_\tau(B_a) = \emptyset$ for $\omega, \tau \in W_r$. 
For each $\omega \in W_r$, we have $f_\omega(B_a) \subset f_\omega(\mathcal{K})$.
Let $m$ denote Lebesgue measure on $\mathcal{K}$. Since the balls are disjoint and contained in $B(x,3r)$, we have 
$$\sum_{\omega \in W_r} m(f_\omega(B_a)) \leq m(B(x,3r)).$$
Using the fact that $|f_\omega(\mathcal{K})| > c_{min}r$, we have
$$M \cdot m(B(x, ac_{min}r)) \leq m(B(x,3r).$$
Hence, $\displaystyle M \leq \frac{m(B(x,3r)}{m(B(x, ac_{min}r))}$.
 Since the ratio compares concentric balls, each with a radius equal to a constant multiple of $r$, we can let $M \leq \left  \lceil {\frac{m(B(x,3r))}{m(B(x,c_{min}r))}} \right \rceil< \infty$, which satisfies the assertion of the proposition.

\end{proof}

\begin{theo}
Let $h,H$ be the unique values such that $P(h) =0 = \bar{P}(H)$.  Then, $h \leq \dim_H(\mathcal{F}) \leq  H$.
\end{theo}
\begin{proof}
Let $\mathcal{U}_n = \{f_{\omega}(\mathcal{K}) : \omega \in W_n\}$.  Notice that $\mathcal{U}_n$ is a cover for all $n \geq 1$.  Hence, by Proposition 4.3, we have
\begin{align*}
\mathcal{H}^H(\mathcal{F}) &= \lim_{\varepsilon \to 0} \inf_{\mathcal{E}} \sum_{E \in \mathcal{E}} |E|^H \leq \lim_{n \to \infty} \sum_{\omega \in W_n} |f_\omega(\mathcal{K})|^H\\
 &\leq \lim_{n \to \infty} \sum_{\omega \in W_n} |\mathcal{K}|^H \bar{c}_\omega^H \leq |\mathcal{K}|^H \cdot L_1 < \infty.
\end{align*}
Thus, $\dim_h(\mathcal{F}) \leq H$.
Let $r>0$ and $B(x,r)$ be a ball centered at $x \in \mathcal{F}$.  By Proposition 4.5, $B(x,r)$ intersects at most $M$ elements of the cover $\mathcal{U}_r$. 
Let $\mathcal{U}_M$ denote the subset of $\mathcal{U}_r$ consisting of all elements that intersect $B(x,r)$ and $W_M$ denote all allowable words associated with an element of $\mathcal{U}_M$.  By Corollary 4.4, we have
\begin{align*}
 \frac{\mu_h(B(x,r))}{r^h} &\leq \frac{ \sum_{f_\omega(\mathcal{K}) \in \mathcal{U}_M} \mu_h(f_\omega(\mathcal{K}))}{r^h} \leq \frac{ \sum_{\omega \in W_M} \frac{L_0}{K_0} c_\omega^h}{r^h} \\
&\leq \frac{M \frac{L_0}{K_0} |\mathcal{K}|^{-h} r^h}{r^h} = M\frac{L_0}{K_0} |\mathcal{K}|^{-h}.
\end{align*}
Hence, $\displaystyle \limsup_{r \to 0} \frac{\mu_h(B(x,r))}{r^h} \leq M\frac{L_0}{K_0} |\mathcal{K}|^{-h} < \infty$.  By the uniform mass distribution principle [4], we have $\mathcal{H}^h(\mathcal{F}) \geq \displaystyle \frac{ M\frac{L_0}{K_0} |\mathcal{K}|^{-h}}{\mu_h(\mathcal{F})} >0$.  Thus, $\dim_H(\mathcal{F}) \geq h$.
\end{proof}

\begin{theo}Let $h,H$ be the unique values such that $P(h)=0=\bar{P}(H)$.  Then, $h \leq \overline{\dim}_B(\mathcal{F}) \leq H$.
\end{theo}
\begin{proof}
The following relationship between Hausdorff and box dimensions is well-known:
$$\dim_H(\mathcal{F}) \leq \underline{\dim}_B(\mathcal{F}) \leq \overline{\dim}_B(\mathcal{F}).$$

Hence, it suffices to show that $\overline{\dim}_B(\mathcal{F}) \leq H$. 
Let $\mathcal{U}_r = \{f_\omega(\mathcal{K}) : |f_\omega(\mathcal{K})| \leq r < |f_{\omega^-}(\mathcal{K})|\}$, $k = \min \{ |\omega| : f_\omega(\mathcal{K}) \in \mathcal{U}_r\}$, and $\mathcal{O}_k = \{ f_\omega(\mathcal{K}) : \omega \in W_k\}$.  
Notice that $\displaystyle \bigcup_{f_\omega(\mathcal{K}) \in \mathcal{U}_r} f_\omega(\mathcal{K}) \subseteq \bigcup_{f_\omega(\mathcal{K}) \in\mathcal{O}_k} f_\omega(\mathcal{K})$.  
Hence, by Proposition 4.4, we have
$$ \sum_{f_\omega(\mathcal{K}) \in \mathcal{U}_r} |f_\omega(\mathcal{K})|^H \leq \sum_{f_\omega(\mathcal{K}) \in \mathcal{O}_k} |f_\omega(\mathcal{K})|^H \leq |\mathcal{K}|^H \sum_{\omega \in W_k} \bar{c}_\omega^H \leq |\mathcal{K}|^H L_1.$$

Also, for $f_\omega(\mathcal{K}) \in \mathcal{U}_r$,
$$|f_\omega(\mathcal{K})| \geq  |f_{\omega^-}(\mathcal{K})|\cdot c_{min} > r c_{min}.$$

Let $N_r(\mathcal{F})$ denote the smallest number of sets of diameter at most $r$ which form a cover of $\mathcal{F}$.  Then,
$$(rc_{min})^H N_r(\mathcal{F}) \leq |f_\omega(\mathcal{K})|^H N_r(\mathcal{F}) \leq \sum_{f_\omega(\mathcal{K}) \in \mathcal{U}_r} |f_\omega(\mathcal{K})|^H \leq |\mathcal{K}|^H L_1.$$

Hence, $N_r(\mathcal{F}) \leq (rc_{min})^{-H}|\mathcal{K}|^H L_1$, and thus

$$ \frac{\log(N_r(\mathcal{F}))}{-\log(r)} \leq \frac{\log(L_1|\mathcal{K}|^H) - H \log(rc_{min})}{-\log(r)} = \frac{\log(L_1|\mathcal{K}|^H)}{-\log(r)} + \frac{H \log(c_{min})}{\log(r)} + H.$$

By the definition of upper box dimension, we have
$$ \overline{\dim}_B(\mathcal{F}) =  \limsup_{r \to 0} \frac{\log(N_r(\mathcal{F}))}{-\log(r)} \leq \limsup_{r \to 0} \left [\frac{\log(L_1|\mathcal{K}|^H)}{-\log(r)} + \frac{H \log(c_{min})}{\log(r)}\right] + H = H.$$

\end{proof}

\begin{remark}  For $E \subset \mathcal{K}$, the following inequalities are well-known:
$$ \dim_H(E) \leq \dim_P(E) \leq \overline{\dim}_B(E) \text{ and } \dim_H(E) \leq \underline{\dim}_B \leq \overline{\dim}_B(E),$$
where $\dim_P(E)$ denote the packing dimension of $E$.  For more information on packing dimension, see [1].
Hence, we have also shown that 
$$ h \leq \dim_P(\mathcal{F}) \leq H \text{ and } h \leq \underline{\dim}_B(\mathcal{F}) \leq H.$$
\end{remark}

\section{Main theorem for Sofic Subshifts}
In this section, we will extend the assertions from Theorems 4.6 and 4.7 to sofic subshifts.  Recall that SFTs are sofic subshifts, but there exist sofic subshifts which cannot be represented as a SFT.  A common charaterization of a sofic shift $(Y, \sigma)$ is that it must be a factor of some SFT, some $(X, \sigma)$.   That is, there exists a continuous map $\psi : X \to Y$ such that $\sigma \circ \psi = \psi \circ \sigma$.

\medskip
We adopt the following definitions from [6].
Let $\mathcal{G} = (G, \mathcal{L})$ be a labeled graph, consisting of a graph $G$ with edge set $\mathcal{E}$ and a labeling $\mathcal{L}: \mathcal{E} \to \mathcal{A}$, where $\mathcal{A}$ is the finite alphabet.  
A subset $X$ of the full shift is a {\it sofic subshift} if $X = X_{\mathcal{G}}$ for some labeled graph $\mathcal{G}$.
Let $\mathcal{G}=(G, \mathcal{L})$ be a labeled graph.  $\mathcal{G}$ is called {\it right-resolving} if for each vertex $v$ in $G$, all edges leaving $v$ have different labels.

\medskip
It is known that every sofic shift has a right-resolving graph presentation [6].  Hence, if $X_{\mathcal{G}}$ is a sofic subshift, we will assume that $\mathcal{G}$ is a right-resolving presentation.  
Notice that $\mathcal{G}$ has $k$ states, which correspond to $k$ total vertices from the graph $G$.  Now, define a $k \times k$ adjacency matrix $M_\mathcal{G}$ by defining the entries as $m_{i,j} = \sum_{e_{i,j}} \mathcal{L}(e_{i,j})$, where $e_{i,j}$ is an edge from vertex $v_i$ to $v_j$ in the graph $G$.   For more information on sofic subshifts and associated graphs, see [6].  We define another $k \times k$ matrix $M_{\mathcal{G},t}$ by defining the entries as $m_{i,j}^{(t)} = \sum_{e_{i,j}} \mathcal{L}(e_{i,j})^t$.

\medskip
Let $\{\mathcal{K}; f_1, \ldots, f_m\}$ be a hyperbolic IFS with $c_i d(x,y) \leq d(f_i(x), f_i(y)) \leq \overline{c}_i d(x,y)$ for $1\leq i \leq m$ and all $x,y \in \mathcal{K}$.  
We will define two $k \times k$ matrices, $A_{\mathcal{G},t}$ and $\overline{A}_{\mathcal{G},t}$ similar to the matrix $M_{\mathcal{G},t}$ above.  Let $A_{\mathcal{G},t}$ be defined by the entries  $a^{(t)}_{i,j} = \sum_{e_{i,j}} c_{(\mathcal{L}(e_{i,j}))}^t$, where $a^{(t)}_{i,j}$ denotes the $(i,j)$-th entry of $A_{\mathcal{G},t}$.  
Let $\overline{A}_{\mathcal{G},t}$ be defined by the entries  $\overline{a}^{(t)}_{i,j} = \sum_{e_{i,j}} \overline{c}_{(\mathcal{L}(e_{i,j}))}^t$.
\begin{lemma} Let $X_\mathcal{G}$ be a sofic subshift, where $\mathcal{G} = (G,\mathcal{L})$.  If $G$ has $k$ vertices, then 
$$ \frac{1}{k} \sum_{i,j =1}^k [A_{\mathcal{G},t}^n]_{i,j} \leq \sum_{\omega \in W_n} c_\omega^t \leq \sum_{i,j=1}^k  [A_{\mathcal{G},t}^n]_{i,j},$$
where $W_n$ denotes all allowable words of length $n$ from $X_\mathcal{G}$.
\end{lemma}
\begin{proof} Let $\omega \in W_n$ for some $n \geq 1$.  Notice that there may be more than one representation for $\omega$ in $\mathcal{G}$.  Since $\displaystyle \sum_{i,j=1}^k [A_{\mathcal{G}}^n]_{ij}$ sums contractive factors related to all labeled paths of length $n$ in $\mathcal{G}$, then $\displaystyle \sum_{\omega \in W_n} c_\omega^t \leq \sum_{i,j=1}^k [A_{\mathcal{G},t}^n]_{ij}$.  Now, if $G$ has $k$ vertices, then $\mathcal{G}$ also has $k$ vertices.  
By assumption, $\mathcal{G}$ is right-resolving, meaning no two edges leaving the same vertex have the same label. Hence, any $\omega \in W_n$ can have at most $k$ representations in $\mathcal{G}$ .  Therefore, for fixed value $k$, $\displaystyle \frac{1}{k}  \sum_{i,j =1}^k [A_{\mathcal{G},t}^n]_{i,j} \leq \sum_{\omega \in W_n} c_\omega^t$.
\end{proof}

\begin{theo}
Let $\{\mathcal{K}; f_1, \ldots, f_m\}$ be a hyperbolic IFS with $c_i d(x,y) \leq d(f_i(x), f_i(y)) \leq \overline{c}_id(x,y)$ for $1 \leq i \leq m$ and all $x,y \in \mathcal{K}$.  Let $X_{\mathcal{G}}$ be a sofic subshift on the alphabet $\mathcal{A} = \{1, \ldots, m\}$ and $\mathcal{F}_{\mathcal{G}}$ be the subfractal defined by the IFS and $X_\mathcal{G}$.  Suppose $A_\mathcal{G}$ is irreducible.  If $\rho(A_{\mathcal{G},h}) = 1 = \rho(\bar{A}_{\mathcal{G},H})$, then 
$$ h \leq \dim_{H}(\mathcal{F}_{\mathcal{G}}) \leq H \text{ and } h \leq \overline{\dim}_B(\mathcal{F}_{\mathcal{G}}) \leq H.$$ 
\end{theo}
\begin{proof}
By Lemma 5.1, we can rewrite the topological lower and upper  pressure function as 
$$P(t) = \lim_{n \to \infty} \frac{1}{n} \log \left (\sum_{\omega \in W_n} c_\omega^t \right ) = \lim_{n \to \infty} \frac{1}{n} \log \left ( \sum_{i,j=1}^k [A_{\mathcal{G},t}^n]_{i,j} \right ) \text{ and }$$
$$\overline{P}(t) =  \lim_{n \to \infty} \frac{1}{n} \log \left ( \sum_{i,j=1}^k [\overline{A}_{\mathcal{G},t}^n]_{i,j} \right ).$$
The remainder of the proof follows as in Theorem 4.6 and Theorem 4.7.
\end{proof}

\begin{remark}
Similar to Remark 2, the values of $h$ and $H$ such that $P(h) =0 = \overline{P}(H)$ also satisfy $\rho(A_{\mathcal{G},h})=1 = \rho(\overline{A}_{\mathcal{G},H})$.
\end{remark}

\begin{remark} 
Similar to Remark 3, due to known relationships between Hausdorff, packing, upper and lower box dimensions, we also have
$$ h \leq \dim_P(\mathcal{F}_\mathcal{G}) \leq H \text{ and } h \leq \underline{\dim}_B(\mathcal{F}_\mathcal{G}) \leq H.$$
\end{remark}

\section{Generalization to reducible matrices}
In this section, we will eliminate the irreducibility condition on the matrices in the case of Hausdorff dimension.  Consider the case where $A_\mathcal{G}$ (or $A_G$ if we have an SFT) is a reducible matrix.
Let $A$ be a reducible $m \times m$ (0,1)- matrix, and $\mathcal{G}$ be the associated graph.  Since $A$ is a reducible matrix, the graph $\mathcal{G}$ is not strongly connected, but it contains a finite number of strongly connected components, say $C_1, \ldots, C_k$.  To each component, we can associate a submatrix $A_1, \ldots A_k$ where $A_i$ is irreducible for $1\leq i \leq k$.  Now, we can simultaneously permute the rows and columns of $A$ to obtain:\\
$$\tilde{A} = \begin{bmatrix} A_k&0&0& \cdots & 0 \\ \ast & A_{k-1} & 0& \cdots &0 \\ \ast & \ast & A_{k-2} & \cdots & 0\\ \vdots & \vdots & \vdots & \ddots & \vdots \\ \ast & \ast & \ast & \cdots & A_1 \end{bmatrix}.$$
For further details on this process, see [6].

\medskip
The process used to obtain $\tilde{A}$ from $A$ will not affect the characteristic polynomial, and hence $A$ and $\tilde{A}$ have the same eigenvalues.  
We can also examine $A^n$ and $\tilde{A}^n$ similarly; that is, by simultaneously interchanging rows and columns of $A$, each entry of $A^n$ will appear in $\tilde{A}^n$, although possibly in a different entry position.  Hence, we can assume $\displaystyle \sum_{i,j=1}^m (A^n)_{ij} = \sum_{i,j=1}^m (\tilde{A}^n)_{ij}$ [6].  Without loss of generality, we will assume that $A$ is in the form of $\tilde{A}$.

\medskip
If $A$ is a reducible $m \times m$ matrix with irreducible components $A_1, \ldots, A_k$, let $A_i$ be an $m_i \times m_i$ matrix for $1\leq i \leq k$.  For $l<k$, we define the set 
$$trn(A_l,A_p) = \{a_{ij} \neq 0: \sum_{s=l+1}^k m_s \leq i \leq \sum_{s=l}^k m_s , \sum_{s=p+1}^k m_s \leq j \leq \sum_{s=p}^k m_s\}$$
 of all non-zero entries from $A_\mathcal{G}$ corresponding to a transitional edge in $\mathcal{G}$ from component $C_l$ associated with $A_l$ to the component $C_p$ associated with $A_p$.  Let $W_{trn}$ denote all finite words in $\Omega_k$ corresponding to a transitional edge from the graph $\mathcal{G}$. 

\medskip
Each strongly connected component of the graph, $C_i$, corresponds to an irreducible submatrix, $A_I$, and a subshift $X_{A_i}$, $1\leq i \leq k$.  For simplicity, we will talk about the construction of words in $X_{A}$ by using the strongly connected components $C_i$, $1 \leq i \leq k$ from $\mathcal{G}$.  Given the structure of the entire graph $\mathcal{G}$ and direction of the transitional edges, words in $X_A$ must begin in a component $C_i$, move through components $C_j$, and end in component $C_l$ where $1\leq i \leq j \leq l \leq k$.  To formalize this in the subshift setting, we introduce the following notation.

\medskip
For $1 \leq i < j \leq k$, let
$$X_{A_i} \circledast X_{A_j} =\{ \omega \in X_A : \omega = \tau a  \xi, \text{ where } \tau \in W_*(A_i), a \in W_{trn}, \xi \in X_{A_j}\}.$$
 Similarly, for $1 \leq i_1 < i_2 < \cdots < i_l \leq k$, we define 
$X_{A_{i_1}} \circledast \cdots \circledast X_{A_{i_l}} = \{ \omega \in X_A: \omega = \tau_1 a_1 \tau _2 \cdots  a_{l-1} \xi, \text{ where } \tau_j \in W_*(A_{i_j}), a_j \in W_{trn}\text{ for } 1 \leq j < l, \xi \in X_{A_{i_l}}\}.$ 

\begin{lemma}
If $\mathcal{G}$ has $k$ irreducible components for $k \geq 2$, then
$$X_\mathcal{G} = \left ( \bigcup_{i=1}^k X_{A_i} \right ) \cup \left ( \bigcup_{j=2}^k \hspace{.2in} \bigcup_{i_1, \ldots, i_j =1}^k X_{A_{i_1}} \circledast \cdots \circledast X_{A_{i_j}} \right),$$
where $i_l < i_{l+1}$ for $1 \leq l < j$.
\end{lemma}
\begin{proof}
 We will use induction for this argument.  If $\mathcal{G}$ has two strongly connected components, $C_1$ and $C_2$, with at least one transitional edge from $C_1$ to $C_2$ then it follows that $X_{A_1} \cup X_{A_2} \cup (X_{A_1} \circledast X_{A_2})\subseteq X_{A_\mathcal{G}}$. 
 Now, let $\omega \in X_{A_\mathcal{G}}$.  Then, $\omega$ must begin in either $C_1$, $C_2$, or on a transitional edge.  
If $\omega$ starts in $C_2$, then $\omega \in X_{A_2}$ because there are no transitional edges leaving $C_2$ in $\mathcal{G}$.  
If $\omega$ starts on a transitional edge, then $\omega \in (X_{A_1} \circledast X_{A_2})$ because it is of the form $\omega = \tau * a * \xi$ where $\tau$ is the empty word from $W_*(A_1)$.  
If $\omega$ starts in $C_1$, then either $\omega \in X_{A_1}$ or $\omega \in (X_{A_1} \circledast X_{A_2})$.  Hence, we must have 
$$X_{A_1} \cup X_{A_2} \cup (X_{A_1} \circledast X_{A_2})= X_{A_\mathcal{G}}.$$
Now, assume $\mathcal{G}$ is a connected graph with $k$ strongly connected components, and consider the subgraph, say $\mathcal{G}|_{(k-1)}$, consisting of the first $k-1$ components.  Assume that 
$$ X_{{\mathcal{G}|_{(k-1)}}} = (\bigcup_{i=1}^{k-1} X_{A_i}) \cup (\bigcup_{j=2}^{k-1} \bigcup_{i_1, \ldots i_j =1}^{k-1} X_{A_{i_1}} \circledast \cdots \circledast X_{A_{i_j}}).$$
By comparing the graphs $\mathcal{G}$ and $\mathcal{G}|_{(k-1)}$ and their corresponding subshifts $X_{\mathcal{G}}$ and $X_{{\mathcal{G}|_{(k-1)}}}$, we can conclude that any word in $X_{\mathcal{G}} - X_{{\mathcal{G}|_{(k-1)}}}$ will end in $C_k$.  Hence,
$$X_{A_\mathcal{G}} = X_{A_{\mathcal{G}_{k-1}}} \cup X_{A_k} \cup \left (\bigcup_{i_1, \ldots i_j=1}^{k-1} X_{A_{i_1}} \circledast \cdots \circledast X_{A_{i_j}} \circledast X_{A_k}\right),$$
which satisfies the assertion.
\end{proof}

\begin{prop}
Let $A_\mathcal{G}$ be a reducible matrix with irreducible components $A_1, \ldots, A_k$.  Then, 
$$ h_{i_j} \leq \dim_H(\mathcal{F}_{X_{A_{i_1}} \circledast \cdots \circledast X_{A_{i_j}}}) \leq H_{i_j},$$
where
$$ h_{i_j}\leq  \dim_H(\mathcal{F}_{X_{A_{i_j}}}) \leq H_{i_j}$$
and $h_{i_j}$, $H_{i_j}$ are the bounds from Theorem 5.2.
\end{prop}
\begin{proof}
Consider a finite word $\tau_1a_1 \tau_2 a_2 \ldots \tau_{j-1}a_{j-1}$ where $\tau_l \in W_*(A_{i_l})$ for $1\leq l \leq j-1$ and $a_l \in trn(A_{i_l}, A_{i_l+1})$.  For any $n \geq 1$, 
 there are finitely many words $\tau_l \in W_*(A_{i_l})$ with $\ell(\tau_l) \leq n$, for all $1 \leq l \leq j$. 
 Hence, there are finitely many words of the form $\tau_1 a_1 \cdots \tau_{j-1}a_{j-1}$ of length $n$.  
So, the collection $S = \{\tau_1 a_1 \cdots \tau_{j-1}a_{j-1} : \tau_l \in  W_*(A_{i_l}) \text{ for } 1\leq l \leq j-1,  a_l \in trn(A_{i_l}, A_{i_l+1}), \ell(\tau_1 a_1 \cdots \tau_{j-1}a_{j-1} ) < \infty \}$ is at most countable since $W_*(A_{i_l})$ is countable for $i \leq i_l \leq k$. 
For $\omega \in S$, let $\omega X_{A_{i_j}} = \{ \omega \xi \in X_{A_{\mathcal{G}}} : \xi \in X_{A_{i_j}}\}$. 
 Then, $\dim_H((\mathcal{F}_{X_{A_{i_11}} \circledast \cdots \circledast X_{A_{i_j}}}) = \displaystyle \sup_{\omega \in S  } \dim_H(\mathcal{F}_{\omega X_{A_{i_j}}})$.\\
First, notice that $\mathcal{F}_{\omega X_{A_i}} = \{f_{\omega \xi} (x) : \xi \in X_{A_i} \text{ and } x \in \mathcal{K}\}$ for any $1 \leq i \leq k$.  
Recall that $f_{\omega \xi}(x) = f_{\xi} \circ f_{\omega}(x)$ and $f_\omega(x) \in \mathcal{K}$ for all $x \in \mathcal{K}$.  Hence, $\mathcal{F}_{\omega X_{A_i}} \subseteq \mathcal{F}_{\omega_{A_i}}.$  Hence,
$$ \dim_H(\mathcal{F}_{\omega X_{A_i}}) \leq \dim_H(\mathcal{F}_{X_{A_i}}) \leq H_i,$$
where $H_i$ is the bound from Theorem 5.2.\\
Let $\omega \in S$ with $\ell(\omega) = m$ and $A_i$ be an irreducible block in $A$. Consider $\dim_H(\mathcal{F}_{\omega X_{A_i}})$.  
Although $\omega X_{A_i}$ is not necessarily a subshift itself, we can apply similar techniques used to prove Theorem 5.2 to show that the zero of the lower topological pressure function $P_{\omega, i}(t) = \lim_{n \to \infty} \frac{1}{n} \log \left ( \sum_{\tau \in W_n(\omega X_{A_i})} c_\tau^t \right )$ is a lower bound for $\dim_H(\mathcal{F}_{\omega X_{A_i}})$.
Notice that 
\begin{align*}
P_{\omega, i} (t) &= \lim_{n \to \infty} \frac{1}{n} \log \left ( \sum_{\tau \in W_n(\omega X_{A_i})} c_\tau^t \right ) 
= \lim_{n \to \infty} \frac{1}{n} \log \left ( \sum_{\tau \in W_{n-m}(X_{A_i})} c_\omega^t c_\tau^t \right ) \\
&= \lim_{n \to \infty} \frac{1}{n} \left [ \log(c_\omega^t) + \log \left( \sum_{\tau \in W_{n-m}(X_{A_i})} c_\tau^t \right ) \right ]\\
 &= \lim_{n \to \infty} \frac{1}{n-m} \log \left ( \sum_{\tau \in W_{n-m}(X_{A_i})} c_\tau^t \right )\\
&= P_{i}(t),
\end{align*}
where $P_{i}(t)$ is the lower topological pressure function associated with the subfractal $\mathcal{F}_{X_{A_i}}$.  Hence, 
$$ h_i \leq \\dim_H(\mathcal{F}_{\omega X_{A_i}}).$$ 
Thus, 
$$\dim_H(\mathcal{F}_{X_{A_{i_11}} \circledast \cdots \circledast X_{A_{i_j}}}) = \displaystyle \sup_{\omega \in S  } \dim_H(\mathcal{F}_{\omega X_{A_{i_j}}}) \leq  H_{i_j} \text{ and }$$
$$ h_{i_j} \leq \dim_H(\mathcal{F}_{X_{A_{i_11}} \circledast \cdots \circledast X_{A_{i_j}}}) ,$$
where $h_{i_j}$ and $H_{i_j}$ are the zeros of the upper and lower topological pressure functions $P_{i_j}(t)$ and $\overline{P}_{i_j}(t)$ with respective the subfractal $\mathcal{F}_{X_{A_{i_j}}}$ for some $1 \leq i_j \leq k$.
\end{proof}
For a similar statement about subshifts with a reducible matrix $A$, we have, by Lemma 6.1,
$$ \mathcal{F}_{X_{A_\mathcal{G}}}= \left ( \bigcup_{i=1}^k \mathcal{F}_{X_{A_i}} \right ) \cup \left ( \bigcup_{j=2}^k  \hspace{.1 in}\bigcup_{1\leq i_1 < \cdots < i_j \leq k}\mathcal{F}_{ X_{A_{i_1}} \circledast \cdots \circledast X_{A_{i_j}}} \right).$$
Thus, by Proposition 6.2, we have the following theorem.
\begin{theo}
Let $X_{A_\mathcal{G}}$ be a sofic subshift associated with matrix $A_\mathcal{G}$.  Assume $A_\mathcal{G}$ has irreducible components $A_1, \ldots, A_k$. 
 Let $\mathcal{F}_{X_{A_\mathcal{G}}}$ and $\mathcal{F}_{X_{A_i}}$ denote the sub-fractals associated with the subshifts $X_{A_{\mathcal{G}}}$ and $X_{A_i}$, respectively.  Then,
 $$\max_{1\leq i \leq k} \{h_i\} \leq \dim_H(\mathcal{F}_{X_{A_\mathcal{G}}}) \leq \max_{1\leq i \leq k} \{ H_i \},$$
where $P_i(h_i)= 0 = \overline{P}_i(H_i)$ given in Theorem 5.2 for all $1 \leq i \leq k$.
\end{theo}

\medskip
\noindent
{\bf Acknowledgements} 
The author acknowledges ND-EPSCoR and NSF grant 1355466 for funding her research.  She thanks her advisor, Do\u{g}an \c{C}\"omez for his guidance and many helpful conversations during the construction of this paper.  She is grateful to Michael P. Cohen for his edits and especially useful comments.  She also thanks Mrinal Kanti Roychowdhury and Azer Akhmedov for their comments on an earlier preprint.

\end{document}